\documentclass{amsart}
\usepackage{amsmath,amssymb,amsthm}
\usepackage[dvipsnames]{xcolor}
\usepackage{bbm}
\usepackage[colorlinks=true, urlcolor=blue,citecolor=OliveGreen,menucolor=]{hyperref}

\DeclareMathOperator{\ad}{\textnormal{ad}}
\DeclareMathOperator{\Ad}{\textnormal{Ad}}
\DeclareMathOperator{\Ran}{\textnormal{Ran}}

\DeclareMathOperator{\id}{\textnormal{id}}

\newcommand{\C}{\mathbb{C}}
\newcommand{\Z}{\mathbb{Z}}
\newcommand{\A}{\mathfrak{A}}
\newcommand{\K}{\mathfrak{k}}
\renewcommand{\H}{\mathcal{H}}
\newcommand{\Id}{\mathbbm 1}
\newcommand{\g}{\mathfrak g}
\renewcommand{\k}{\mathfrak k}

\newcommand{\M}{\mathcal{M}}

\newcommand{\N}{\mathcal{N}}

\newcommand{\lgk}{\mathcal{A}(G,K)}
\newcommand{\lcgk}{\mathcal{AC}(G,K)}

\newcommand{\J}{\mathcal{J}}

\newtheorem*{theorem*}{Theorem}
\newtheorem{theorem}{Theorem}

\newtheorem{proposition}[theorem]{Proposition}

\theoremstyle{definition}
\newtheorem{definition}[theorem]{Definition}

\newtheorem{remark}[theorem]{Remark}
\numberwithin{theorem}{section}
\numberwithin{equation}{section}

\begin{document}

\title{Nijenhuis operators on homogeneous spaces related to $C^*$-algebras}

\author{Tomasz Goli\'nski}
\address{University of Bia{\l}ystok, Faculty of Mathematics\\
    Ciołkowskiego 1M, 15-245 Bia{\l}ystok, Poland}
\email{tomaszg@math.uwb.edu.pl}

\author{Gabriel Larotonda}
\address{Universidad de Buenos Aires and CONICET, Argentina}
\email{glaroton@dm.uba.ar}

\author{Alice Barbora Tumpach}
\address{Institut CNRS Pauli, Wien, Technische Universität Wien, and Laboratoire Paul Painlev\'e, Villeneuve d'Ascq}
\email{barbara.tumpach@math.cnrs.fr}

\begin{abstract}For a unital non-simple $C^*$-algebra $\A$ we consider its Banach--Lie group $G$ of invertible elements. For a given closed ideal $\K$ in $\A$, we consider the embedded Banach--Lie subgroup $K$ of $G$ of elements differing from the unit element by an element in $\K$. We study vector bundle maps of the tangent space of the homogeneous space $G/K$, induced by an admissible bounded operator on $\A$. In particular, we discuss when this vector bundle map is a Nijenhuis operator in $G/K$. The special case of almost complex structures in $G/K$ is also addressed. Examples for particular classes of $C^*$-algebras are presented, including the Toeplitz algebra and crossed products by $\Z$. 
\end{abstract}


\maketitle

\section{Introduction}

The aim of this work is to apply the results of \cite{GLT-NN} to the study of certain homogeneous spaces related to unital non-simple $C^*$-algebras. In the aforementioned  paper, we consider the homogeneous space $G/K$ where $G$ is a Banach--Lie group and $K$ an embedded (not necessarily split) Banach--Lie subgroup  of $G$.  We study homogeneous vector bundle maps $\N:T(G/K)\to T(G/K)$, which are induced by admissible bounded operators on the Banach--Lie algebra $\g$ of the group $G$. In particular we compute the Nijenhuis torsion and derive the conditions for $\N$ to be a Nijenhuis operator.

Nijenhuis operators are useful in many areas of mathematics and mathematical physics, including the integrability of almost complex structures \cite{eckmann51,newlander,beltita05integrability} or the study of integrable systems, see e.g. \cite{magri84,magri90,bols-bor} and the references therein. For example they describe certain deformations of Poisson brackets on the manifold via so called Poisson--Nijenhuis structures with further links to Poisson--Lie groups \cite{magri90}.
They are also studied separately, within the framework of so-called Nijenhuis geometry, see e.g. the series of recent papers \cite{bolsinov-n1,bolsinov-na5} and references therein. A nice review of the story around Nijenhuis torsion can be read in \cite{kosmann-bial}.

In this work we consider a $C^*$-algebra $\A$ with a unit element $\Id$ and a fixed closed two-sided ideal $\K$. In this situation the quotient space $\A/\k$ is also a $C^*$-algebra. The group $G$ is chosen to be the open set $\A^\times$ of all invertible elements in $\A$ and the subgroup $K$ is chosen to be $(\Id+\k)^\times$. It follows that $G/K$ is a homogeneous space and even more, it is also a Banach--Lie group, since the subgroup is normal, see Proposition~\ref{prop:quotient_group}. One can also consider the group $(\A/\k)^\times$ of invertible elements in the quotient $C^*$-algebra and the relationship between $(\A/\k)^\times$ and $\A^\times/(\Id+\k)^\times$ is presented in Proposition~\ref{prop:isomorphism}.

To illustrate the situation we propose three classes of possible admissible operators on $\A$. The first class is obtained by considering a bounded functional $\ell$ vanishing on the ideal $\K$, and its extensions to $\A$ by means  of the Hahn--Banach theorem; this leads to a rank one admissible operator in $\A$. The second class of operators  is given by left and right multiplication by elements from $\A$, and the last class is given by the adjoint operator $\ad_d$ for some $d\in\A$.

These possible ways of obtaining Nijenhuis operators are applied to four classes of unital non-simple $C^*$-algebras. The first two are the most standard examples: bounded operators on a separable  Hilbert space and continuous functions $C(X)$ on a compact set $X$. We also describe the situation in the case of the Toeplitz algebra of the unit circle $\mathbb T$ with continuous symbols, and in the case of crossed products of $C(X)$ and $\Z$ by a homeomorphism $\alpha:X\to X$.

Other possible applications of the results of the paper \cite{GLT-NN} to the study of unitary orbits are considered in \cite{GLT-kahler}.

The structure of this paper is as follows. Section~\ref{sec:basic} is devoted to the basic notions and facts, which were proven in \cite{GLT-NN}. No proofs were included in this text, but they can be found in the referenced paper. Subsequently Section~\ref{sec:homog} deals with $C^*$-algebras and proves the basic properties of homogeneous spaces in this context. It also contains subsections detailing possible classes of admissible operators on $\A$. The last section, Section~\ref{sec:cstar}, discusses particular classes of $C^*$-algebras.

\section{Basic results}\label{sec:basic}

This section introduces the notation and reiterates the results of the paper \cite{GLT-NN}. This presentation is abbreviated by necessity, but tries to be self contained. However all required details and proofs can be found in the aforementioned paper.

First of all if $f:\M_1\to \M_2$ is a smooth map between two smooth manifolds, its differential will be denoted by $f_*:T\M_1\to T\M_2$, or if there is a need to specify the point $m\in\M_1$, it will be denoted $f_{*m}:T_m\M_1\to T_{f(m)}\M_2$. Sometimes we will omit specifying the point when its clear from the context. By immersion (resp. submersion) we will mean a map such that the differential at every point is an injection with closed image (resp. a surjection).

Let $G$ be a real Banach--Lie group with Banach--Lie algebra $\g$. 
\begin{itemize}
    \item The left and right multiplication by elements $g\in G$ will be denoted as $l_g(h):=gh$ and $r_g(h):=hg$.
    \item The differential $(l_g)_{*1}$ will be denoted as $L_g$ and $R_g = (r_g)_{*1}$. 
    \item The adjoint map on the Lie algebra (the differential of conjugation) is denoted as $\Ad_g$ i.e. $\Ad_g= L_gR_{g^{-1}}$. 
    \item The Lie bracket in $\g$ will be denoted as $[v,w]=\ad_v w$, where $\ad=(\Ad)_{*1}$.
\end{itemize}

\begin{definition}[Homogeneous spaces]\label{homs} Let $K$ be an immersed Banach--Lie subgroup of $G$ with Banach--Lie subalgebra $\k\subset \g$. We say that 
$G/K$ is a \emph{homogeneous space} of $G$ if the quotient space for the right action of $K$ on $G$
has a Hausdorff Banach manifold structure such that the quotient map $\pi: G \rightarrow G/K$, $g\mapsto gK$ is a submersion.
\end{definition}

In the quotient $G/K$ there is a distinguished point, called the base point, which we will denote as $p_0=\pi(K)$. Moreover the group $G$ acts in a natural way on $G/K$ as follows:
\[\alpha:G\times G/K\to G/K \qquad \alpha(g,p)=\pi(gh)\]
for $p=\pi(h)\in G/K$. For a fixed $g\in G$, we denote by $\alpha_g$ the mapping $\alpha_g(p)=\alpha(g,p)$.

Next we discuss homogeneous vector bundle maps on the tangent bundle $T(G/K)$ of the homogeneous space. In particular we are interested in the ones that come from a bounded linear map defined at the level of the Banach--Lie algebra $\mathfrak g$. All vector bundle maps in this paper are always covering identity, i.e. they map each fiber to itself.

\begin{definition}[Homogeneous bundle maps]\label{hls}A smooth vector bundle map $\N:T(G/K)\to T(G/K)$ covering identity is called \emph{homogeneous} if it is invariant with respect to the action by automorphisms~$\alpha_g$:
\[
    (\alpha_g)_{*}\, \N_p=\N_{\alpha_g(p)}(\alpha_g)_{*} \qquad \textit{ for any }p\in G/K \textit{ and any }g\in G.
\]
\end{definition} 
By homogeneity, any such map comes from some $\N_{p_0}\in \mathcal B(T_{p_0}(G/K))$ at the base point $p_0$. 

Now, we look at the situation from the side of the Lie algebra $\mathfrak g$ of the group $G$. Namely we are interested in a set of bounded operators on $\g$, which descend to a homogeneous bundle map $\N:T(G/K)\to T(G/K)$.

\begin{definition}[Admissible operators]\label{homj}The set of admissible operators $\lgk$ on $\mathfrak g$ is defined as the set of $N\in  \mathcal B(\mathfrak g)$ satisfying the following conditions
\begin{equation}\label{eq:admissible1}
N\k\subset \k,
\end{equation}
\begin{equation}\label{eq:admissible2}
\Ran(\Ad_k N-N\Ad_k)\subset \k,
\end{equation}
where $\mathcal B(\g)$ denotes the set of linear bounded operators on $\g$.
\end{definition}

Condition \eqref{eq:admissible1} ensures that $N$ correctly defines an operator $\N_{p_0}\in \mathcal B(T_{p_0}(G/K))$ at the base point $p_0$. The condition \eqref{eq:admissible2} ensures on the other hand that $\N_{p_0}$ can be propagated to the whole tangent bundle $T(G/K)$. Summarizing we get:

\begin{proposition}\label{prop:induced_homog_struct}
An admissible operator $N\in \lgk$ induces
the homogeneous vector bundle map $\N:T(G/K) \to T(G/K)$ given at each $p=\pi(g)\in G/K$ by
\begin{equation}\label{Np}
\N_p = (\alpha_g)_* \N_{p_0} \, (\alpha_g)_*^{-1},
\end{equation}
where $\N_{p_0}: T_{p_0}G/K \rightarrow T_{p_0}G/K$ is defined as 
\begin{equation}\label{Np0}
\N_{p_0}\pi_{*1}v:=\pi_{*1}N v
\end{equation}
for $v\in\mathfrak g$.
\end{proposition}


\begin{definition}[Nijenhuis torsion]Let $\M$ be any smooth manifold and let $\N:T\M\to T\M$ be a smooth vector bundle map. The \emph{Nijenhuis torsion} of $\N$ is defined as 
\[
\Omega_{\N}(X,Y)=\N [\N X,Y]+\N [X,\N Y]-[\N X,\N Y]-\N^{\; 2}[X,Y]
\]
for $X,Y$ vector fields in $\mathcal M$. We say that $\N$ is a \emph{Nijenhuis operator} in $\mathcal M$ if its torsion vanishes.
\end{definition}

It is known that the value of $\Omega_\N$ at a given point actually depends only on the values of the vector fields in that point. For $\N$ induced by an admissible operator $N$, it would be more convenient for our purpose to express $\Omega_\N$ in terms of an admissible operator $N$.

\begin{theorem}\label{torsionJ}Let $G/K$ be a homogeneous space and consider an admissible operator $N\in\lgk$. Let $X,Y$ be vector fields on $G/K$. Let $p=\pi(g)$ and choose $v,w\in \mathfrak g$ such that $X(p)=(\alpha_g)_*\pi_{*1}v$ and $Y(p)=(\alpha_g)_*\pi_{*1}w$. Then the Nijenhuis torsion of $\N$ can be expressed as
$$
\Omega_{\N}(X,Y)(p)=(\alpha_g)_*\pi_{*1}\big(N[v,N w]+N[N v,w]-[N v,N w]-N^2[v,w]\big).
$$
In particular $\N$ is a Nijenhuis operator if and only if
\begin{equation}\label{localtor}
N[v,N w]+N[N v,w]-[N v,N w]-N^2[v,w]\in \mathfrak k
\end{equation}
for all $v,w\in\mathfrak g$.
\end{theorem} 

Let us now present the special case when $\N$ is an almost complex structure.

\begin{definition}[Homogeneous almost complex structure]A \emph{homogeneous almost complex structure} is a homogeneous vector bundle map $\J$ in $G/K$ (Definition \ref{hls}) with the additional requirement that $\J^2=-1$. 
\end{definition}

To get an almost complex structure out of the admissible linear bounded operators on $\g$ we need to consider the following.

\begin{definition}\label{admJ}
Define the following subset of admissible operators on $\mathfrak g$
\[ \lcgk = \{ J\in \lgk \;|\; \Ran(J^2+1)\subset \mathfrak k\}.\]
\end{definition}

It is easy to observe that if $\J$ is induced by $J \in \lcgk$ it satisfies the condition $\J_{p_0}^2=-1$ in $T_{p_0}(G/K)$. By homogeneity it also holds in the whole tangent bundle $T(G/K)$.

By the Newlander--Nirenberg theorem (\cite{eckmann51,newlander}, in the real analytic Banach context proved in \cite[Theorem 7]{beltita05integrability}) an almost complex structure $\J$, for which the Nijenhuis torsion vanishes, is integrable, i.e. there exists a structure of complex Banach manifold on $G/K$ compatible with $\J$.

\section{Homogeneous spaces and admissible operators related to $C^*$-algebras}\label{sec:homog}
\subsection{General considerations}

\begin{definition}[$C^*$-algebra]
A \emph{$C^*$-algebra} is a complex Banach space $\A$ equipped with a linear multiplication and an anti-linear involution denoted by $*$ such that
\[ 
\|ab\| \leq \|a\| \|b\| ,\qquad (ab)^* = b^*a^*,\qquad  \|aa^*\|=\|a\|^2
\]
for all $a,b\in\A$. We say the $C^*$-algebra is unital if it contains a unit, which we will denote by $\Id$.
\end{definition}

The model example of a $C^*$-algebra is the set of all bounded operators $\mathcal B(\H)$ on a Hilbert space $\H$. In fact, due to the Gelfand--Naimark--Segal construction, every $C^*$-algebra can be seen as a closed $*$-subalgebra in $\mathcal B(\H)$ for some Hilbert space $\H$. 

In a unital $C^*$-algebra the set of invertible elements in $\A$ is open. Moreover since the product and the inverse map are smooth in the $C^*$-topology, the set of invertible operators is a Banach--Lie group, which we will denote by $G=\A^\times$. Naturally, its Banach--Lie algebra is $\g=\A$. Let $\exp:\A\to\A$ denote the exponential map of $\A$, defined as a power series, this map coincides with the Lie group exponential of $G=\A^{\times}$. 

Assume there exists a closed two-sided non-trivial ideal $\K$ in the $C^*$-algebra $\A$ (i.e. the algebra $\A$ is not simple). In this case $\K$ is automatically preserved by the involution $*$, see e.g. \cite[5.37]{douglas}. Moreover the quotient space $\A/\K$ always possesses a structure of a $C^*$-algebra.  

Fixing the ideal $\K$ we consider a subgroup $K$ of $\A^{\times}$ defined as the set of invertible elements which differs from the identity by an element in $\k$. Note that it is indeed a subgroup. To see this consider an element $\Id+k\in K$ and its inverse $a\in \A^\times$. From $a(\Id+k) = \Id$ it follows that $a\in \Id + \K$. Thus $K$ is actually equal to $\A^\times\cap (\Id + \K)$, which for brevity we will also denote as $(\Id+\K)^\times$,
\[ K = \A^\times\cap (\Id + \K) = (\Id + \K)^{\times}.\]

Let us denote by $Q:\A\to \A/\mathfrak k$ the linear quotient map and by $\pi:G\to G/K$ the group quotient map. 

We will be interested in the study of $G/K$ and homogeneous vector bundle maps on it. Note that in general $\K$ might not possess a complementary subspace in $\A$.

\begin{proposition}\label{prop:quotient_group}
The quotient space $G/K$ is a homogeneous space of $G$. Moreover it has a structure of Banach--Lie group with Lie algebra $\A/\k$ such that for the Lie group exponentials we have
\begin{equation}\label{natur-exp}
\exp_{G/K}\circ \, Q = \pi \circ \exp_G=\pi \circ \exp.
\end{equation}
\end{proposition}
\begin{proof}
Since $\K$ is a two sided-ideal of $\A$, the subgroup $K$ is always normal. Indeed for any $g \in \A^\times$ we have
\[g\big(\A^\times\cap (\Id + \K)\big)g^{-1} = 
\A^\times\cap (gg^{-1} + g\K g^{-1}) = \A^\times\cap (\Id + \K).
\]
Moreover since $\K$ is closed, the subgroup $K$ is automatically embedded with respect to the norm topology.

The equality \eqref{natur-exp} follows from \cite[Theorem II.2]{glockner-neeb03} and the previous observations. In particular $\pi_{*1}=Q$ and  $\pi:G\to G/K$ is a smooth submersion. This shows that $G/K$ is a homogeneous space according to Definition~\ref{homs}. 
\end{proof}

Note that the Banach--Lie algebras of $G/K=\A^\times/(\Id+\k)^\times$ and of $(\A/\mathfrak k)^{\times}$ are both isomorphic to $\A/\mathfrak k$. 

\begin{proposition}\label{prop:isomorphism}
Let $\iota:\A^\times/(\Id+\k)^\times\to (\A/\k)^{\times}$ be defined as $\iota(\pi(a))=Q(a)=a+\mathfrak k$. Then $\iota$ is a Banach--Lie group isomorphism onto an open Banach--Lie subgroup of $(\A/\k)^{\times}$, and the identity components of $\A^\times/(\Id+\k)^\times$ and $(\A/\k)^{\times}$ are isomorphic (via $\iota$) as Banach--Lie groups. 
\end{proposition}
\begin{proof}
The map $\iota$ does not depend on the choice of the representative since the relation $\pi(a) = \pi(b)$ is equivalent to the fact that there exists $g = \Id + k\in (\Id+\k)^\times$ with $k\in\k$, such that $b = ag$, hence
$$ \iota\big(\pi(b)\big) = ag + \k = a(\Id+k) + \k = a + ak +\k = a + \k = \iota\big(\pi(a)\big).$$

Let us verify that the map $\iota$ takes values in $(\A/\k)^{\times}$, i.e. that the class $a+\k$ is invertible in $(\A/\k)^{\times}$ for $a\in \A^\times$. It follows from the observation that $$(a^{-1}+\k)\cdot (a+\k) = \Id +\k = Q(\Id).$$

It is easy to see that $\iota$ is a group homomorphism. Indeed
$$
\iota\big(\pi(a)\big)\cdot \iota\big(\pi(b)\big) = (a + \k)(b+\k) = ab + a\k + \k b + \k = ab + \k = \iota\big(\pi(ab)\big).
$$

To show that $\iota$ is injective, we now only need to check that the only preimage of the unit $\Id_{ (\A/\k)^{\times}} = \Id+ \k $ in $(\A/\k)^{\times}$ is the unit in $G/K = \A^\times/(\Id+\k)^\times$. This is clear since 
$$\iota\big(\pi(a)\big)  = \Id + \k \Leftrightarrow a + \k = \Id + \k$$
implies $\pi(a) = \Id_{G/K}.$

We claim that 
\begin{equation}\label{iotasmooth}
\exp_{ (\A/ \mathfrak k)^{\times}}=\iota\circ \exp_{G/K}
\end{equation}
To prove this, denote with $\widetilde{\exp}$  the exponential map of the $C^*$-algebra $\A/\mathfrak k$, which is also the exponential map $\exp_{ (\A/ \mathfrak k)^{\times}}$ of the Banach--Lie group $(\A/\mathfrak  k)^{\times}$. If $k\in \k$, then for any $a\in\A$ we have $(a+k)^n=a^n+k'$ for some $k'\in \k$, hence $\widetilde{\exp}(Q(a))=Q(\exp(a))$. Thus
\[
\exp_{(\A/ \mathfrak k)^{\times}}(Q(a)) = Q(e^a)=e^a+\mathfrak k.
\]
On the other hand since $\pi_{*1}=Q$ and by means of the naturality of exponential maps of Banach--Lie groups \eqref{natur-exp}
\[
\iota \circ \exp_{G/K}(Q(a))=\iota \circ \pi \circ \exp(a)=\iota(\pi(e^a))=e^a+\mathfrak k,
\]
and this proves \eqref{iotasmooth}. From there we can conclude two things: the first one is that $\iota$ is smooth (see for instance \cite[Lema 3.3.1]{larotonda-geom-book}). The second one is that $\iota_{*1}=\id$ is an isomorphism. Since $\iota$ is injective, the conclusion follows from \cite[Proposici\'on 3.3.4]{larotonda-geom-book}.
\end{proof}

\begin{remark}\label{connected}
    If $\A^{\times}$ and/or $(\A/\mathfrak k)^{\times}$ are connected, the situation simplifies considerably. This happens for instance if $\A$ and/or $\A/\mathfrak k$ are von Neumann algebras: write $g=pu$ the polar decomposition of $g$ in the von Neumann algebra $\A$, with $p=\sqrt{g^*g}$ positive invertible and $u=p^{-1}g$ unitary; both $p,u\in \A$. 
    Since the unitary group of a von Neumann algebra is connected, see \cite[5.29]{douglas}, it remains to show that the positive cone is connected as well. Let $x=x^*\in\A$ be the unique self-adjoint logarithm of $p$ in $\A$. A curve connecting $p$ to $\Id$ is then $e^{tx}$, $t\in [0,1]$. 
\end{remark}

\begin{remark}
Since the identity connected components of $G/K$ and $(\A/\k)^\times$ are isomorphic, so are the tangent bundles to these components. In consequence every vector bundle map (including Nijenhuis operators) on $T(G/K)$ defines a vector bundle map on the identity component of $T(\A/\k)^\times$.
\end{remark}

The easiest example of a $C^*$-algebra is $\A=\mathcal B(\H)$. This situation was presented in \cite[Examples 5.2 and 5.3]{GLT-NN} and will be discussed in this paper in Subsection~\ref{ssec:bounded}. Before presenting other examples in Section~\ref{sec:cstar}, we will consider three classes of vector bundle maps, which would be useful in the sequel.

\subsection{Rank one admissible operator}\label{ssec:rank1}

This construction is a generalization of \cite[Example 5.2]{GLT-NN} to the case of $C^*$-algebra $\A$.

First we define a linear functional $\ell$ on the subspace $S=\C \Id +\K\subset\A$.
Since $\K$ is closed and doesn't contain $\Id$, we observe that the distance $d=dist(\Id,\K)=\inf_{k\in \K} \|k+\Id\|>0$. In consequence $S$ is closed:
\[
\|t \Id + k \|=|t| \|\Id +k'\| \ge |t| d.
\]
We define $\ell$ on $S$
in the following way:
\[ \ell(\K)=0,\qquad \ell(\Id)=1.\] 
It follows that $\ell$ is bounded in $S$. By means of  Hahn--Banach theorem one extends it to a bounded functional on the whole $C^*$-algebra $\A$. This extension will also be denoted by $\ell$. 



Fix $n\in \mathcal B(\A)$ such that $n\notin \k$ and consider a bounded operator $N\in\mathcal B(\A)$ given by
\begin{equation}\label{rank1} N(a) = \ell(a) \cdot n. \end{equation}
Since $\ell$ vanishes on $\mathfrak k$, the condition $N(\mathfrak k)\subset \mathfrak k$ is trivially satisfied. Let $g=\Id+k\in K$, $g^{-1}=\Id+\tilde{k}$ with $k,\tilde{k}\in\mathfrak k$. Then for any $a\in\A$
\[
\Ad_g a = gag^{-1}=a+ka+a\tilde{k}+ka\tilde{k}\in a +\mathfrak k.
\]
Then
\[
\Ad_g N(a) - N(\Ad_g a) = \ell(a) n + k_1 - \ell(a+k_2)n = k_1 \in \k.
\]
Therefore $N\in\lgk$ and thus it gives rise to a homogeneous vector bundle map $\N$ on $G/K$. 

Now condition \eqref{localtor} of Theorem~\ref{torsionJ} tells us that the bundle map generated by $N$ is Nijenhuis if and only if for all $v,w\in \A$ we have
\[
\ell(w)\ell([v,n])+\ell(v)\ell([n,w])-\ell([v,w])\ell(n)=0.
\]
Substituting $w=\Id$ it implies that $\ell([v,n])=0$ for all $v\in \A$, and this in turn implies that $\ell([v,w])\ell(n)=0$ for all $v,w\in \A$. 

There are two possibilities now. If $\ell(n)\ne 0$, $\N$ is a Nijenhuis operator iff
\begin{equation}\label{tor-rank1}
\ell([v,w])=0\qquad \textrm{for all }v,w\in \A.
\end{equation}
On the other hand if we assume that $\ell(n)=0$, the condition is
\begin{equation}\label{tor-rank1n}
\ell([v,n])=0\qquad \textrm{for all }v\in \A.
\end{equation} 
For simplicity we will focus in this paper on the first case, taking for example $n=\Id$.

Let us also mention that vector bundle maps $\N$ of the discussed form never give rise to almost complex structures since they are never invertible. At each fiber the codimension of kernel is equal to $1$.

\subsection{Left and right multiplication}\label{ssec:lr}

Consider another possible choice of the operator $N$ on $\A$: namely right and left multiplication by bounded operators:
\[ N(a)= AaB, \]
for $A,B,a\in \A$. In this case the condition \eqref{eq:admissible1} is automatically satisfied since $\K$ is a two-sided ideal in $\A$.

The condition \eqref{eq:admissible2} for $N$ to be admissible written out explicitly takes the form
\[
(\Id+k)AaB(\Id+\tilde k) - A(\Id+k)a(\Id+\tilde k)B \in \K
\]
for all $k\in\K$ such that $\Id+k\in \A^\times$ and all $a\in\A$, where $\Id+\tilde k = (\Id+k)^{-1}$ with $\tilde{k}\in\K$. Again just like in the previous example, we notice that, after expanding, the terms without $k$ or $\tilde k$ cancel out and all others belong to $\K$. 
Thus the operator $N$ is always admissible and, due to Proposition~\ref{prop:induced_homog_struct}, it descends to the operator $\N$ on the homogeneous space. 

Note that if we choose $A$ and $B$ such that $A^2$, $B^2$ are in the center of $\A$ and moreover $A^2B^2=-\Id$, we get an almost complex structure on $G/K$.

Let us verify using Theorem~\ref{torsionJ} when $\N$ is a Nijenhuis operator. The condition \eqref{localtor} doesn't hold in general (except obviously in the commutative case). However in a simpler case when either $A=\Id$ or $B=\Id$ it is always satisfied. Thus left (or right) multiplication by an element from $\A$ always gives rise to a Nijenhuis operator on the homogeneous space $G/K$. If we choose this operator in  such a way that it's square is $-\Id$, we obtain an integrable complex structure on $G/K$.


\subsection{Adjoint action}\label{ssec:ad}

Now, we define $N$ as $\ad_d$ for some $d\in \A$. It is an admissible operator $N\in\lgk$ if and only if for all $a\in\A$ and $k\in\K$
\[ [k,d] \in \K\] and 
\[ [a,[k,d]] \in \K.\]
Naturally, both conditions are satisfied since $k$ belongs to the ideal $\K$.

The condition \eqref{localtor} for the induced vector bundle map $\N$ to be Nijenhuis in this case can be written as
\[\label{tor-ad}
\big[ [d,a] , [d,b] \big] \in \K
\]
for all $a,b\in\A$. It is satisfied in the commutative case and if $d\in\K$, however in those situations $\N=0$.

It is difficult to say more about $\N$ without specifying the $C^*$-algebra $\A$, so more results will be presented in the next section.


\section{Applications to particular $C^*$-algebras}\label{sec:cstar}
In this paper we will focus on the following (classes of) $C^*$-algebras:
\begin{enumerate}
    \item the $C^*$-algebra of bounded operators $\mathcal B(\H)$ on a separable Hilbert space $\H$
    \item the set of continuous functions $C(X)$ on a compact set $X$ 
    \item Toeplitz algebra
    \item Crossed product algebra related to a dynamical system  $C(X)\rtimes_\alpha \Z$
\end{enumerate}
In the first case the $C^*$-algebra is commutative, which leads to a much simpler structure. Notably every unital commutative $C^*$-algebra is of this form by Gelfand--Naimark theorem.




We will present the characterization of ideals in each case and discuss a possible examples of homogeneous bundle maps, Nijenhuis operators and almost complex structures.

\subsection{$C^*$-algebra of bounded operators}\label{ssec:bounded}

\subsubsection{Calkin algebra.} Let's discuss briefly the most common example of a $C^*$-algebra, i.e. the set of all bounded operators $\mathcal B(\H)$ on a Hilbert space $\H$. This case was discussed in \cite{GLT-NN}. It is known that the only closed two-sided ideal of $\mathcal B(\H)$ for a separable Hilbert space $\H$ is the ideal of compact operators $\mathcal K(\H)$. For this reason there is only one case to discuss here $\k=\mathcal K(\H)$. In this case the $C^*$-algebra $\mathcal B(\H)/\mathcal K(\H)$ is known as the Calkin algebra. 

The invertible elements in the Calkin algebra correspond to Fredholm operators in $\H$. Observe that since compact perturbations do not change the Fredholm index, the Fredholm index is also well defined on the Calkin algebra. Let us describe the map $\iota$ presented in Proposition~\ref{prop:isomorphism} explicitly in this particular case:
$$
\begin{array}{lcll}
\iota: & G/K = \mathcal B(\H)^\times/\big(\Id + \mathcal K(\H)\big)^\times & \longrightarrow & \left(\mathcal B(\H)/\mathcal K(\H)\right)^\times\\
& a\big(\Id + \mathcal K(\H)\big)^{\times} & \longmapsto & a + \mathcal K(\H).
\end{array}
$$
Thus the image of the map $\iota$ consists only of elements of Fredholm index $0$ and thus $\iota$ is not a surjection.

\begin{remark}First we note that $\mathcal B(\H)^{\times}$ is connected since $\mathcal B(H)$ is a von Neumann algebra (Remark \ref{connected}). On the other hand, connected components of the set  of Fredholm operators $\mathcal F(\H)$ are exactly the operators of index $k$, for each $k\in\mathbb Z$ \cite[Theorem 5.36]{douglas}. In particular the operators of index $0$ are exactly the identity component of the group of invertible operators in the Calkin algebra. Hence, by Proposition \ref{prop:isomorphism}, the map $\iota$ is a Banach--Lie group isomorphism between $G/K$ and the identity component of the group of invertible elements in the Calkin algebra.
\end{remark}



\subsubsection{Rank one admissible operators on the Calkin algebra.} We now recall that Halmos showed in \cite[Theorem~1]{halmos-commutators} that any self-adjoint operator is a sum of two commutators of the form $[A, A^*]$, which implies that  the set of all commutators of operators in $\mathcal B(\H)$ spans the whole $\mathcal B(\H)$. Then, the homogeneous vector bundle maps induced by admissible operators on $\mathcal B(\H)$ of rank 1 defined in Subsection~\ref{ssec:rank1} for $\ell(n)\neq 0$ are never Nijenhuis, since condition \eqref{tor-rank1} implies that $\N$ is zero, which contradicts its definition.

Let us consider for a while the case $\ell(n)=0$. Take as $n$ any non-trivial orthogonal projection $P$ on $\H$. The closed space $\mathcal B_{\textnormal{co}}(\H)$ generated by all commutators with $P$ consists of block co-diagonal operators with respect to the decomposition $\H = P\H \oplus (\Id-P)\H$. We can now amend the definition of $\ell$ by making it vanish on the closed space spanned by $\mathcal K(\H)$, $\mathcal B_{\textnormal{co}}(\H)$ and $P$. In such a way the condition \eqref{tor-rank1n} is satisfied and we do obtain a Nijenhuis operator. This case is related to the restricted Grassmannian, see \cite{segal} and also the discussion in \cite[Example 5.4]{GLT-NN}.

\subsubsection{Left and right multiplication on the Calkin algebra}

As mentioned above, left (or right) multiplication by an element from $\mathcal B(\H)$ always gives rise to a Nijenhuis operator on the Calkin algebra.

One can also obtain simple examples of complex structures on the group $G/K$ by considering a decomposition of $\H$ into two orthogonal infinite dimensional subspaces $\H=\H_+\oplus \H_-$. Denote by $A$ an isomorphism of $\H_+$ and $\H_-$ extended to a bounded operator (partial isometry) on $\H$. Then the operator $iA - iA^*$ squares to $-\Id$ and thus multiplication by it (left or right) gives us a complex structure.

\subsubsection{Adjoint operators on the Calkin algebra.} Consider $N$ be defined as the adjoint operator $\ad_d$ for some $d\in\mathcal B(\H)$. It is Nijenhuis iff the following condition holds
\[ \big[[d,a],[d,b]\big] \in \mathcal K(\H) \]
for all $a,b\in\mathcal B(\H)$.
Observe that such an operator would never induce an almost complex structure as it would require
\[ \big[d,[d,x]\big] \in -x + \mathcal K(\H) \]
for all $x\in\mathcal B(\H)$. It would mean that $Q(x)=x+\mathcal K(\H)$ is a commutator of elements in the Calkin algebra. Take for an example $x=-\Id$. From the argument presented in \cite{wintner} it follows that such an operator cannot be equal to any commutator of elements in any Banach algebra.

\subsection{$C^*$-algebra of  continuous functions on a compact set}\label{ssec:continuous}

\subsubsection{Ideals of $C(X)$} In this subsection we will consider $C^*$-algebras of the form $C(X)$, where $X$ is a compact Hausdorff space. It follows from Gelfand--Naimark theorem (see e.g. \cite[4.29]{douglas}) that actually every commutative unital $C^*$-algebra is of this form.

Closed ideals of $C(X)$ are easy to describe, see e.g. \cite[Exercise 2.2]{douglas}. Namely all closed ideals are given by closed subsets $Y\subset X$ in the following way:
\begin{equation}\label{kY}
\k_Y = \{ f \in C(X) | f_{|Y} = 0 \}.
\end{equation}
Note that since compact Hausdorff sets are normal topological spaces, Tietze extension theorem implies that any continuous map on $Y$ can be extended to a continuous map on $X$. Consequently the quotient $C^*$-algebra $C(X)/\k_Y$ is isomorphic to $C(Y)$. Thus in this case the quotient Banach--Lie group $G/K=C(X)^\times/(\Id+\k_Y)^\times$ is actually isomorphic to $C(Y)^\times$ via the map $\iota$ described in Proposition~\ref{prop:isomorphism}. 
Note that the group $(\Id+\k_Y)^\times$ is split if and only if $Y$ is open and closed in $X$.

Moreover due to commutativity all vector bundle maps (not necessarily induced by admissible operators) are Nijenhuis operators.

\subsubsection{Rank one admissible operators on $C(X)$}Let us give an example of Nijenhuis operator induced by an admissible operator on $C(X)$ of rank 1 (as defined in Subsection~\ref{ssec:rank1}, equation~\eqref{rank1}). Consider the  functional $\ell:C(X)\to \C$ given by an evaluation functional $ev_y$ at a point $y\in Y$. In this case the map $N$ (and $\N$ as well) maps a function $f$ to the constant function $f(y)$. More generally,
by Riesz--Markov--Kakutani representation theorem, every functional $\ell:C(X)\to \C$ vanishing on $\k_Y$ and equal to one on $\Id$ is given by the integral with respect to a probabilistic measure with support in $Y$. Hence any Nijenhuis operator in this case is of this form.

\subsubsection{Left and right multiplication on $C(X)$}
The Nijenhuis operators defined via (left or right) multiplication in Subsection~\ref{ssec:lr} act as multiplication by a function $f$ from $C(Y)$ on $TC(Y)^\times\cong C(Y)^\times \times C(Y)$. They can be chosen to define a complex structure if the function $f$ is constant equal to $\pm i$ on each connected component of $Y$. 

\subsubsection{Adjoint operator on $C(X)$}
The operators defined via $\ad_d$ in Subsection~\ref{ssec:ad} are all trivially zero.

\subsection{Toeplitz algebra}\label{ssec:toeplitz}

\subsubsection{Toeplitz operators on Hardy space} Consider the complex Hilbert space $L^2(\mathbb T)$, where $\mathbb T$ denotes the unit circle in $\C$. Denote by $H^2$ the Hardy space, i.e. the closed subspace consisting of functions which can be extended holomorphically to the interior of the unit disc. Alternatively, one can consider the canonical basis of $L^2(\mathbb T)$ consisting of the functions $\chi_n$, $n\in\Z$, defined as $\chi_n(z) = z^n$, $z\in\mathbb T$. The Hardy space $H^2$ can then be seen as a closed space generated by $\{\chi_n\}_{n\geq 0}$. Let $P\in\mathcal B(L^2(\mathbb T))$ be the orthogonal projection onto $H^2$.

\begin{definition}
Let $\varphi\in C(\mathbb T)$. The \emph{Toeplitz operator} $T_\varphi$ on $H^2$ is defined as $T_\varphi f = P(\varphi f)$. The \emph{Toeplitz algebra} $\mathfrak T$ is the $C^*$-algebra generated by all Toeplitz operators.
\end{definition}

We will again consider as the ideal $\k$ the set of compact operators $\k = \mathcal K(H^2)$. 
It is known that $\k\subset \mathfrak T$ and moreover it is equal to the commutator ideal of $\mathfrak T$, see \cite[7.12]{douglas}, \cite[3.5.10]{murphy}. As a consequence the quotient $C^*$-algebra $\mathfrak T/\k$ is abelian and moreover it is $*$-isomorphic to $C(\mathbb T)$ (see e.g. \cite[3.5.11]{murphy}).

\begin{proposition}
The homogeneous space $G/K=\mathfrak T^\times/(\Id+\k)^\times$ with the natural structure of Banach--Lie group modeled on $C(\mathbb T)$ is isomorphic to the set of elements in $C(\mathbb T)^\times$ which have winding number $0$. In consequence, this group is also abelian. 
\end{proposition}

\begin{proof}
The claim follows from Proposition~\ref{prop:isomorphism} by noting two things. First that $\mathfrak T^{\times}$ is connected: for the Toeplitz operator $T_{\varphi}$ to be invertible, it must be $\varphi=e^{\psi}$ for some $\psi\in C(\mathbb T)$, see \cite[3.5.15]{murphy}. Thus $\alpha(t)=T_{e^{t\psi}}$ connects $T_{\varphi}$ to the identity $\Id$.
On the other hand the connected components of $C(\mathbb T)^\times$ are indexed by the winding number. 
\end{proof}

This fact simplifies the problem of describing Nijenhuis operators. Namely every vector bundle map on $T(G/K)$ is automatically a Nijenhuis operator (regardless of the fact if it is induced by an admissible operator in $\lgk$). Moreover in the effect, since $G/K$ is a connected component of $C(\mathbb T)^\times$ it reduces the situation to the one described in Subsection~\ref{ssec:continuous}. It is more straightforward to describe homogeneous bundle maps directly on the homogeneous space $C(\mathbb T)^\times$, instead of considering admissible operators on $\mathfrak T$.

\subsubsection{Rank one admissible operators on Toeplitz algebra}

The functional $\ell$ on $\mathfrak T$ defined in Subsection~\ref{ssec:rank1} descends to the quotient $C^*$-algebra $\tilde \ell:\mathfrak T/\mathcal K(H^2)\cong C(\mathbb T)\to \C$. The operator $\N$ is given then by $\N_{p_0}(f) = \tilde\ell(f) \tilde n$ for $\tilde n = \pi_{*\Id}(n)\in C(\mathbb T)$. The functionals $\tilde \ell$ are determined by a choice of a probability measure on $\mathbb T$, just like in the case of the $C^*$-algebras of continuous functions.

\subsubsection{Left and right multiplication on Toeplitz algebra} Consider admissible operators on $\mathfrak T$ defined by left and right multiplication (see Subsection~\ref{ssec:lr}) by some elements $A,B\in \mathfrak T$. Since the commutator of two elements of Toeplitz algebra is always compact, we observe that
\[ N(a) = AaB = ABa + k \]
for some $k\in \mathcal K(H^2)$. Thus a homogeneous bundle map induced by left and right multiplication $a\mapsto AaB$ is the same as the one induced by multiplication from the left $a\mapsto ABa$.

\subsubsection{Adjoint operators on the Toeplitz algebra} Let us focus for a second on admissible operators on $\mathfrak T$ defined by the adjoint operator $\ad_d$ for some $d\in \mathfrak T$. From the fact that the commutator ideal of $\mathfrak T$ is equal to $\k$, we conclude that they induce only trivial zero Nijenhuis operators.



\subsection{Crossed product $C^*$-algebras related to dynamical systems}\label{ssec:crossed}

\subsubsection{Covariant representation and crossed product}
Let $X$ be a compact Hausdorff space and let $\alpha:X\to X$ be a homeomorphism. We will call the pair $(X,\alpha)$ a dynamical system.




Consider the Hilbert space $\H=l^2(X\times \Z)$. On this Hilbert space one can introduce the representation of $C(X)$ by multiplication operators 
\[ (f \cdot a)(x,n) = f(x) a(x,n)\]
and the representation of $\Z$ by composition with $\alpha$ on $X$ and by group multiplication on $\Z$, i.e.
\[ (m \cdot a)(x,n) = a(\alpha^m(x),n+m) \]
for $a\in \H$, $x\in X$, $n,m\in\Z$, $f\in C(X)$. Note that in this section we use the convention that the powers of $\alpha$ are considered with respect to operation of composition, e.g. $\alpha^2 = \alpha\circ\alpha$. 

To simplify the notation we will identify elements of $C(X)$ with multiplication operators acting on $\H$ and we can also define an operator $u\in\mathcal B(\H)$, which induces the action of $\Z$. Since the action is given by a bijection, the operator $u$ is unitary.

Note that the representation of $C(X)$ and the unitary $u$ satisfy
\begin{equation}\label{cov}
 ua = \alpha(a) u \quad \textrm{for }a\in C(X),
\end{equation}
where the action of $\alpha$ on $C(X)$ is defined by precomposition
\[ \alpha(a)(x) = a(\alpha(x)). \]
A pair of representations of $C(X)$ and $\Z$ satisfying \eqref{cov} is called a \emph{covariant representation} of the dynamical system $(X,\alpha)$.

Consider the $C^*$-algebra $\A=C^*(C(X),u)=C(X)\rtimes_\alpha \Z$ generated by the representation of $C(X)$ and the unitary $u$. This is a $C^*$-algebra generated by a covariant representation of the dynamical system $(X,\alpha)$. Since the group under consideration is amenable it is the \emph{crossed product} of $C(X)$ and the group $\Z$ with respect to the homeomorphism $\alpha$, see e.g. \cite{tomiyama-notes}.

Note that for an arbitrary group, the crossed product is the $C^*$-envelope of the algebra of finite formal sums. 
The general theory of crossed products is a very vast subject, see e.g. \cite{pedersen,davidson-cstar,williams-crossed} for an overview. It is useful in the study of dynamical systems, also in the case when $\alpha$ might not be invertible or is only partially defined, see e.g. \cite{tomiyama-notes,giordano1995,silvestrov02,exel03,antonevich-crossed,kwasniewski16,beltita-neeb16,bardadyn24}. 

Moreover if the homeomorphism $\alpha$ is topologically free (i.e. the set of fixed points has empty interior), then one can consider even simpler faithful representation of the crossed product. Assume that there exists an $\alpha$-invariant probabilistic measure $\mu$ on $X$ with $\operatorname{supp}\mu = X$ and let $\tilde\H=L^2(X,\mu)$. In this case we will identify the commutative $C^*$-algebra $C(X)$ with multiplication operators in $\mathcal B(\tilde\H)$ and introduce the operator $\tilde u$ as follows
\[ \tilde uf = f \circ \alpha \]
for $f\in\tilde \H$. Since the measure $\mu$ is $\alpha$-invariant, the operator $\tilde u$ is unitary. It is also a covariant representation of the dynamical system $(X,\alpha)$. In this case the $C^*$-algebra generated by $C(X)$ and $\tilde u$ is isomorphic to the crossed product $C(X)\rtimes_\alpha \Z$.


\subsubsection{Ideals of $C(X)\rtimes_{\alpha} \mathbb{Z}$}
Assume that 
there exists a closed $\alpha$-invariant set $Y\subset X$. 
Define the $C^*$-algebra $\tilde\k_Y=C^*(\k_Y,u)$ generated by $\k_Y$ defined in \eqref{kY} and $u$. It turns out that $\tilde \k_Y$ is a proper ideal in $C(X) \rtimes_\alpha \mathbb Z$, see e.g. \cite[Proposition VIII.3.3]{davidson-cstar}. The ideal $\tilde k_Y$ can be identified with $k_Y\rtimes_\alpha \Z$, see \cite[Lemma 3.17]{williams-crossed}. Moreover following \cite[Corollary 3.20]{williams-crossed} we conclude that the quotient $C^*$-algebra $\big(C(X) \rtimes_\alpha \mathbb Z\big) / \tilde \k_Y$ can be identified with $C(Y)\rtimes_{\alpha_{|Y}} \Z$.

It is easy to see moreover that choosing another $\alpha$-invariant subset $Y'\subset Y$ we get the ideal $\tilde\k_{Y'}\supset \tilde\k_Y$. In this way we have a whole lattice of ideals generated by $\alpha$-invariant sets with lattice operations given by union and intersection of invariant sets, which encodes the dynamic of $\alpha$.

The description of the groups of invertible is an open problem and partial results for particular classes of crossed product algebras are known, for example the problem of density of $\A^\times$ is related to stable rank of crossed product, see e.g. \cite{poon89}.

Let us now discuss homogeneous bundle maps induced by admissible operators $N\in\lgk$ on the $C^*$-algebra $C(X)\rtimes_\alpha \Z$.
First, observe that if $Y$ consists only of fixed points of $\alpha$, i.e. $\alpha$ is identity on $Y$, the crossed product $C(Y)\rtimes_{\alpha_{|Y}} \Z$ trivializes to $C(Y)$. In this situation all homogeneous bundle maps are Nijenhuis, like mentioned in Subsection~\ref{ssec:continuous}. 
For the remainder of this subsection we will be assuming that $Y$ does not consists only of fixed points of $\alpha$.

\subsubsection{Rank one admissible operators on the crossed product $C(X)\rtimes_{\alpha} \mathbb{Z}$}

We consider now homogeneous bundle maps constructed via the functional $\ell$ in Subsection~\ref{ssec:rank1}. The functional $\ell$ needs to vanish on $\tilde \k_Y$. By linearity and continuity we need to define it on the elements of the form $f u^k$, for $f\in C(X)$, $k\in\Z$. Using the same argument as in Subsection~\ref{ssec:continuous}, we conclude that $\ell(f u^k)$ is given by the integral with respect to a finite measure $\mu_k$ such that $\operatorname{supp}\mu_k\subset Y$. Thus the functional $\ell$ is uniquely determined by a series of measures $\{\mu_k\}_{k\in\Z}$ such that $\mu_0$ is additionally probabilistic (to ensure that $\ell(\Id) = 1$).

Let us investigate when an operator $\N$ defined via $\ell$ is Nijenhuis, for $\ell(n)\neq 0$. From condition \eqref{tor-rank1} using linearity and continuity again we see that it is equivalent to
\[[fu^k,gu^n]=(f(g\circ\alpha^k) - g(f\circ\alpha^n))u^{n+k}\in \tilde\k_Y\] 
for all $f,g\in C(X)$ and $k,n\in \Z$. Taking as a special case $g=\Id\in C(X)$ (i.e. the constant function equal $1$) and $n=1$, we note that the condition requires $f(y)-f(\alpha(y)) = 0$ for all $f\in C(X)$ and $y\in Y$. That implies that $\N$ is never Nijenhuis.

\subsubsection{Left and right multiplication on the crossed product $C(X)\rtimes_{\alpha} \mathbb{Z}$}For the admissible operators given by left and right multiplication, see Subsection~\ref{ssec:lr}, note that we already demonstrated that left multiplication and right multiplication considered separately are always Nijenhuis. Thus an operator e.g. of the form 
\[ f u^k \longmapsto u (f u^k) = \alpha(f) u^{k+1} \]
is a Nijenhuis operator on the space $G/K$.

\subsubsection{Adjoint operators on the crossed product $C(X)\rtimes_{\alpha} \mathbb{Z}$} Consider an admissible operator $N=\ad_d$, $d\in\C(X)\rtimes_\alpha\Z$, described in Subsection~\ref{ssec:ad}. Let us choose $d=u$. In this case we have
\[ N(f u^k) = [u, fu^k] = (\alpha(f) - f)u^{k+1}. \]
This operator acts on the coefficients in formal sums as a difference operator with respect to the homeomorphism $\alpha$. In order to verify if it is Nijenhuis we use the condition \eqref{tor-ad}. It is not difficult to see that this condition cannot be satisfied since we have 
\[
\big[[u,f],[u,g]\big] = \big[ (\alpha(f) - f)u, (\alpha(g) - g)u \big] = \]
\[\big((\alpha(f) - f)(\alpha^2(g) - \alpha(g))- (\alpha(g) - g)(\alpha^2(f) - \alpha(f))\big) u^2,
\]
and it is easy to obtain functions $f$ and $g$ such that the coefficient function in front of $u^2$ in $\big[[u,f],[u,g]\big]$ would not vanish at some $y\in Y$. In consequence vector bundle maps induced by $\ad_u$ are not Nijenhuis.


\subsection{Summary of the results}

Let us present two tables to visually summarize some of the results. The first table shows whether the discussed admissible operators give rise to Nijenhuis operators for the discussed classes of $C^*$-algebras. In some cases the question is still open.
\begin{table}[!h]
\centering
\begin{tabular}{|c|c|c|c|} \hline
     & Rank one & LR mult. & adjoint  \\ \hline
 bounded    & -/+ & ?/+ & ?  \\ \hline
 continuous & + & + & 0  \\ \hline
 Toeplitz & + & + & 0  \\ \hline
 crossed (fixed) & + & + & 0  \\ \hline
 crossed (non-fixed) & - & ?/+ & -/?  \\ \hline
\end{tabular}
\caption{Existence of Nijenhuis operators: in each case the existence of Nijenhuis operators is marked by a $+$-sign, the non-existence of Nijenhuis operators is marked by a $-$-sign, and the fact that all Nijenhuis operators are identically zero is marked by a $0$. A question marked is used when the existence of Nijenhuis operators is an open question for some particular situation (see the corresponding sections).}
\end{table}

The second table is meant to visualize the possibility of finding an almost complex structure within the described classes. One open question still remains.
\begin{table}[!h]
\centering
\begin{tabular}{|c|c|c|c|} \hline
     & Rank one & LR mult. & adjoint  \\ \hline
 bounded    & - & + & -  \\ \hline
 continuous & - & + & -  \\ \hline
 Toeplitz & - & + & -  \\ \hline
 crossed (fixed) & - & + & -  \\ \hline
 crossed (non-fixed) & - & + & ?  \\ \hline
\end{tabular}
\caption{Possibility of almost complex structures: in each case the fact that an almost-complex structure can be constructed using the considered class of Nijenhuis operators is marked by a $+$-sign, when no almost-complex structure can be constructed we use a $-$-sign, and when the question is still open we use a $?$-sign.}
\end{table}

\section*{Acknowledgments}

This research was partially supported by ANPCyT, CONICET, UBACyT 20020220400256BA (Argentina) and 2020 National Science Centre, Poland/Fonds zur Förderung der wissenschaftlichen Forschung, Austria grant ``Banach Poisson--Lie groups and integrable systems'' 2020/01/Y/ST1/00123, I-5015N.

The authors are grateful to Bartosz Kwaśniewski for his advice and help while writing this paper.


\end{document}